\documentclass[twoside]{amsart}%
\usepackage{amssymb}
\usepackage{amsmath,amsthm}
\usepackage{amsfonts}
\usepackage{graphicx}
\usepackage{amsmath}
\usepackage[all]{xy}
\usepackage{array}
\usepackage{xcolor}
\usepackage{multicol}
\usepackage{hyperref}

\usepackage{mathabx}
\usepackage{leftidx}

\DeclareMathOperator{\type}{type}
\DeclareMathOperator{\dom}{dom}

\providecommand{\U}[1]{\protect\rule{.1in}{.1in}}
\newtheorem{theorem}{Theorem}
\newtheorem{mtheorem}{Main Theorem}

\newtheorem{claim}[theorem]{Claim}

\newtheorem{corollary}[theorem]{Corollary}

\newtheorem{definition}[theorem]{Definition}

\newtheorem{lemma}[theorem]{Lemma}

\newtheorem{problem}[theorem]{Problem}
\newtheorem{proposition}[theorem]{Proposition}
\newtheorem{remark}[theorem]{Remark}

\def\genby #1{\langle #1 \rangle}

\date{}

\begin{document}


\title{Forcing a Basis into $\aleph_1$-Free Groups}

\author{Daniel Bossaller, Daniel Herden, Alexandra V. Pasi}
\address{
Department of Mathematics,
Baylor University,
Sid Richardson Building,
1410 S.4th Street,
Waco, TX 76706, USA}
\email{daniel\_bossaller@baylor.edu, alexandra\_pasi@baylor.edu}
\email[Corresponding author]{daniel\_herden@baylor.edu}


\begin{abstract}
In this paper, we address the question of when a non-free $\aleph_1$-free group $H$ can be be free in a transitive cardinality-preserving model extension. Using the $\Gamma$-invariant, denoted $\Gamma(H)$, we present a necessary and sufficient condition resolving this question for $\aleph_1$-free groups of cardinality~$\aleph_1$. Specifically, if $\Gamma(H) = [\aleph_1]$, then $H$ will be free in a transitive model extension if
and only if $\aleph_1$ collapses, while for $\Gamma(H) \ne [\aleph_1]$ there exist cardinality-preserving forcings that will add a basis to $H$.
In particular, for $\Gamma(H) \neq [\aleph_1]$, we provide a poset $(\mathcal P_{\rm pb}, \leq)$ of partial bases for adding a basis to $H$ without collapsing~$\aleph_1$.
\end{abstract}

\maketitle

\section{Introduction}

$\aleph_1$-free groups, abelian groups whose countable subgroups are free, are objects of significant algebraic complexity, and the class of $\aleph_1$-free groups exhibits a high degree of diversity. One standardized test of the algebraic complexity of a class is ring realization: Dugas, G{\"o}bel \cite{DG} and Corner, G{\"o}bel \cite{CG} showed that any ring with free additive structure can be realized as the endomorphism ring of some $\aleph_1$-free group. This can be interpreted as a strong statement of algebraic complexity: almost any property that does not outright contradict $\aleph_1$-freeness will be realized by some $\aleph_1$-free group. In particular, we can construct arbitrarily large $\aleph_1$-free groups $A$ such that $\rm{End}(A) \cong \mathbb{Z}$, indicating that the class of $\aleph_1$-free groups is significantly different from the class of free groups, for which every endomorphism is uniquely determined by the images of its basis elements (and thus the endomorphism ring of a rank $\kappa$ free group with $\kappa \ge \aleph_0$ will have size $2^{\kappa}$). For $\aleph_1$-free groups $H$ of size $\aleph_1$, we can characterize the relationship between $\aleph_1$-freeness and freeness algebraically using the $\Gamma$-invariant, denoted $\Gamma(H)$, introduced by Eklof in \cite {Eklof77}. Given an $\aleph_1$-filtration $\bigcup_{\alpha < \aleph_1} H_\alpha$ of $H$, \[\Gamma(H) := \big[\{\alpha < \aleph_1 \mid H/H_{\alpha} \text{ is not } \aleph_1 \text{-free} \}\big], \] where $[\cdot]$ designates the equivalence class defined by intersection with closed unbounded subsets (clubs) of $\aleph_1$. Under this definition, $H$ is free if and only if $\Gamma(H)=[\emptyset]$.

Yet set-theoretically, $\aleph_1$-free groups can be thought of as somewhat simple objects. One important set-theoretic property of $\aleph_1$-freeness which demonstrates this is its absoluteness, see \cite{HP}. In particular, if an abelian group $H$ is $\aleph_1$-free in some transitive model of ZFC, then $H$ is $\aleph_1$-free in any transitive model of ZFC containing $H$. One consequence of the absoluteness of $\aleph_1$-freeness is the following set-theoretic characterization of the relationship between $\aleph_1$-freeness and freeness: an abelian group $H$ is $\aleph_1$-free in a countable transitive model of ZFC if and only if it is free in some model extension. This result is easily obtained by performing a forcing which collapses the cardinality $|H|$ to be countable.

It is properties such as this absoluteness that make $\aleph_1$-free groups particularly susceptible to set-theoretic techniques, and $\aleph_1$-free groups are a common object of independence results. For instance, $\aleph_1$-free groups are instrumental for Shelah's celebrated resolution of the Whitehead problem \cite{Eklof76,S1}. In the wake of Shelah's proof that the existence of non-trivial $\aleph_1$-free Whitehead groups is undecidable from ZFC, set-theoretic prediction principles have become widely accepted tools for $\aleph_1$-free constructions \cite{EM,GT}.

In this paper, we explore the question of when an $\aleph_1$-free abelian group $H$ can be forced to be free. As mentioned, such a forcing can be achieved by collapsing $|H|$ to be countable. However, much of the structure of the ground model is lost in such a cardinal collapse, and we would like to know under what conditions forcing is possible without collapsing $|H|$. For $\aleph_1$-free groups of cardinality $\aleph_1$, we show that such a forcing is possible if and only if $\Gamma(H)\ne [\aleph_1]$. In particular, for $\Gamma(H)\ne [\aleph_1]$, we can add a basis to $H$ by forcing a club into a suitable stationary subset of~$\aleph_1$. Adding clubs to stationary sets and their complements is a well-known task \cite{AS,BHK,Krueger}. The forcing presented by Baumgartner, Harrington, and Kleinberg \cite{BHK} achieves this while preserving $\aleph_1$, while Abraham, Shelah \cite{AS} preserves all cardinals.

The shortcoming of these set-theoretic constructions is that they achieve freeness without explicitly constructing a basis for $H$. To address this, the latter half of the paper is devoted to a novel forcing poset $(\mathcal P_{\rm pb}, \leq)$ which forces freeness by directly adding an independent generating set to $H$. More precisely, we consider the set
\[ \mathcal P_{\rm pb} = \{p \subseteq H \mid p \text{ is linearly independent}, |p| < \aleph_1, \text{ and } H/\langle p \rangle \text{ is } \aleph_1\text{-free}\} \]
ordered under set inclusion. This culminates in the main result of this paper.

\begin{mtheorem} \label{mtheorem}
Let $\mathbf M$ be a transitive model of {\rm ZFC} with \[ \mathbf M \vDash \text{``$H$ is an $\aleph_1$-free group of size $\aleph_1$.\!''}\]
\begin{itemize}
\item[(a)] If $\Gamma(H) = [\aleph_1]$ in $\mathbf M$, then $H$ is free in a transitive model extension $\mathbf N$ of $\mathbf M$ if and only if $\aleph_1\!\!$ collapses.
\item[(b)] If $\mathbf M$ is countable with $\Gamma(H)\ne [\aleph_1]$, then $H$ will be free in some generic cardinality-preserving forcing extension of $\mathbf M$. Moreover, if $G$ is any $\mathcal P_{\rm pb}$-generic filter, then ${\mathbf M}[G]$ adds a basis to $H$ while preserving $\aleph_1$.
\end{itemize}
\end{mtheorem}

Over the years there have been many applications of forcing techniques to abelian group theory \cite{AGS,BS,MaS,MeS,S2}. However, these methods have been directed primarily at using forcing as a tool for establishing independence results. This paper proposes a shift in viewpoint by rather concentrating on the forcing process itself. Indeed, Theorem~\ref{mtheorem} is a natural result in the sense that it
focuses on $(\mathcal P_{\rm pb}, \leq)$ as the most direct forcing poset which solves the task of adding a basis to an $\aleph_1$-free group $H$. In a similar manner, we can tailor minimally invasive forcing posets to alter other properties of $\aleph_1$-free groups, e.g., forcing a direct summand into an indecomposable group. It is our belief that a systematic investigation of the resulting family of forcings will provide a deeper understanding of
$\aleph_1$-free groups as mathematical objects with a double life in algebra and set theory. In addition, this may also help discover new families of $\aleph_1$-free groups with interesting properties. Theorem~\ref{mtheorem} highlights that groups $H$ with $\Gamma(H) = [\aleph_1]$ behave quite differently than all other $\aleph_1$-free groups of size $\aleph_1$. We have christened these groups \emph{turbid} as they are as far away from being free
as is possible with respect to the club filter.

Our paper is structured as follows. In Section \ref{sec:2}, we review the $\Gamma$-invariant and introduce turbid groups. In Section \ref{sec:3.1}, we will add a basis by forcing a club into a stationary set, while Section \ref{sec:3.3} focuses on forcing with partial bases and lists properties of the forcing poset $(\mathcal P_{\rm pb}, \leq)$. Theorem \ref{mtheorem} is a combination of Lemma \ref{mtheorema}, Lemma \ref{not turbid}, Theorem \ref{freeforce}, and Theorem \ref{MainTheorem}.

\section{Groups with Full $\Gamma$-Invariant} \label{sec:2}

In this section, we review some basic properties of the $\Gamma$-invariant and introduce turbid groups as $\aleph_1$-free groups $H$ of cardinality $\aleph_1$ with $\Gamma$-invariant $\Gamma(H)=[\aleph_1]$.

\subsection{The $\Gamma$-Invariant} \label{sec:2.1}

We will define the $\Gamma$-invariant $\Gamma(H)$ of an $\aleph_1$-free group $H$ of cardinality $\aleph_1$ and relate the freeness of $H$ to $\Gamma(H)$. For further details and proofs, see \cite[Section IV.1]{EM}.

We start with $\aleph_1$-filtrations and a club filter equivalence relation on subsets of~$\aleph_1$.

\begin{definition}\label{filtration}
Let $H$ be an abelian group of cardinality $\aleph_1$. An \emph{$\aleph_1$-filtration} of $H$ is a sequence $\{H_{\alpha}\mid \alpha < \aleph_1\}$ of subgroups of $H$ whose union is $H$ and which satisfies for all $\alpha, \beta < \aleph_1$:
\begin{enumerate}
\item[$(i)$] $|H_\alpha|\le\aleph_0$;
\item[$(ii)$] if $\alpha \leq \beta$, then $H_\alpha \subseteq H_\beta$;
\item[$(iii)$] if $\alpha$ is a limit ordinal, then $H_\alpha=\bigcup_{\beta<\alpha}H_\beta$.
\end{enumerate}

Let $I \subseteq \aleph_1$ be closed and bounded $($i.e., $\sup(J) \in I$ for all $J \subseteq I$ and $\sup(I) < \aleph_1).$ Given an $\aleph_1$-filtration, $\mathcal F = \{H_\alpha \mid \alpha < \aleph_1\}$ of $H$, the \emph{partial filtration of $\mathcal F$ induced by $I$}, $\mathcal F_I$, is $\{H_\alpha \mid \alpha \in I\}$. Since $I \subseteq \aleph_1$ is a set of ordinals, it carries some order type $< \aleph_1$; to this end, we define $\type (\mathcal F_I) := \type (I)$.
\end{definition}

\begin{definition}\mbox{}
\begin{itemize}
\item[$(a)$] If $X$ and $Y$ are subsets of $\aleph_1$, we define $X \sim Y$ if there exists some closed and unbounded set (club) $C\subseteq \aleph_1$ such that $X \cap C=Y \cap C$.
This defines an equivalence relation on $\mathcal{P}(\aleph_1)$. Denote the equivalence class of $X$ by~$[X]$.
\item[$(b)$] For $X,Y \subseteq \aleph_1$, we define $[X] \subseteq [Y]$ if there exists some club $C\subseteq \aleph_1$ with $X\cap C \subseteq Y\cap C$. This defines a partial
order of equivalence classes $[X]$ with least element $[\emptyset]$ and greatest element $[\aleph_1]$.
\end{itemize}
\end{definition}

We now define the $\Gamma$-invariant of an $\aleph_1$-free group of cardinality $\aleph_1$.

\begin{definition}\label{gammainvariant}
Let $H$ be an $\aleph_1$-free
group of cardinality $\aleph_1$, and let $\{H_\alpha\mid \alpha< \aleph_1\}$ be an $\aleph_1$-filtration of $H$. Let
\[E=\{\alpha < \aleph_1\mid H/{H_\alpha} \text{ is not } \aleph_1\text{-free} \}.\] We call the equivalence class of $E$, $[E]$, the \emph{$\Gamma$-invariant} of $H$, denoted $\Gamma(H)$.
\end{definition}

Note that the $\Gamma$-invariant does not depend on the chosen filtration $\{H_\alpha\mid \alpha< \aleph_1\}$ of $H$.
The next result relates the freeness of an $\aleph_1$-free group to its $\Gamma$-invariant.

\begin{theorem} \label{freeGamma}
If $H$ is an $\aleph_1$-free group of cardinality $\aleph_1$, then $H$ is free if and only if $\Gamma(H)=[\emptyset]$.
\end{theorem}

This result and corresponding definitions generalize in a straightforward way to $\leq\!\! \kappa$-generated $\kappa$-free modules $M$, for $\kappa$ a regular uncountable cardinal, with the $\Gamma$-invariant $\Gamma(M)$ defined with respect to a $\kappa$-filtration of $M$.

\subsection{Turbid Groups} \label{sec:2.2}

We have seen that groups $H$ with empty $\Gamma$-invariant $\Gamma(H)=[\emptyset]$ coincide with free groups.
This raises the natural question what happens at the other end of the spectrum, when we have a full $\Gamma$-invariant $\Gamma(H)=[\aleph_1]$.
Intuitively, such groups $H$ will be as far away as possible from being free, and we will call these groups turbid.

\begin{definition}
An $\aleph_1$-free group $H$ of cardinality $\aleph_1$ is called \emph{turbid} if $\Gamma(H) = \left[\aleph_1\right]$.
\end{definition}

We have the following immediate alternative characterization of turbid groups.

\begin{corollary} \label{charturbid}
Let $H$ be an $\aleph_1$-free group of cardinality $\aleph_1$ and let $\{H_\alpha \mid \alpha < \aleph_1\}$ be an $\aleph_1$-filtration of $H$. Then the following holds:
\begin{itemize}
\item[(a)] $H$ is free if and only if $\aleph_1-E$ contains a club.
\item[(b)] $H$ is turbid if and only if $E$ contains a club.
\end{itemize}

\begin{proof}
$H$ is turbid iff $\Gamma(H) = \left[\aleph_1\right]$. This is the case if and only if there exists a club $C$ with $E \cap C = \aleph_1 \cap C = C$, and thus $C \subseteq E$. For part (a) see Theorem~\ref{freeGamma}.
\end{proof}
\end{corollary}

Given any stationary subset $E \subseteq \aleph_1$, \cite[Section IV.0]{EM} gives a construction for some (strongly) $\aleph_1$-free group $H$ of cardinality $\aleph_1$ such that $\Gamma(H) = [E]$. In particular, for $E= \aleph_1$, this provides the existence of an $\aleph_1$-free turbid group. For a more familiar example of a turbid group, one may consult the Baer-Specker group. Recall that the Baer-Specker group is the countable product of copies of $\mathbb Z$, which we will denote by $P = \prod_\omega \mathbb Z$. It is evident that $P$ has size $2^{\aleph_0}$. Moreover, it is well known that $P$ fails to be free but is $\aleph_1$-free \cite{Baer37, specker50}. The interested reader may consult \cite{EM} for a further exploration of $P$. We extend one of those results here.

\begin{lemma}{$(\operatorname{ZFC} + \operatorname{CH})$}
The Baer-Specker group is turbid.
\end{lemma}

\begin{proof}
The assumption of the Continuum Hypothesis assures that $P$ has cardinality $2^{\aleph_0} = \aleph_1$. By \cite[Theorem IV.2.8]{EM}, $P$ fails to be strongly $\aleph_1$-free. In particular, for any filtration $\{H_\alpha \mid \alpha \in \aleph_1\}$ of $P$, the complement $\aleph_1 - E$ is bounded. Thus $\Gamma(P) = [\aleph_1]$ and $P$ is turbid.
\end{proof}
A forthcoming article \cite{BHP2} by the authors will further explore properties of the Baer-Specker group $P = \prod_\omega \mathbb Z$ which follow from its turbidity.

Of interest to our main result Theorem \ref{mtheorem} is the following result that turbid groups cannot be made free in a transitive model extension without cardinal collapse.

\begin{lemma} \label{mtheorema}
Let $\mathbf{M}$ be a transitive model of {\rm ZFC} and $H\in \mathbf{M}$ a turbid group. Suppose that $\mathbf{N}$ is a transitive model of {\rm ZFC} containing $\mathbf{M}$. Then $H$ is free in $\mathbf{N}$ if and only if $\aleph_1^{\mathbf{N}} \neq \aleph_1^{\mathbf{M}}$. Moreover, if $\aleph_1^{\mathbf{N}} = \aleph_1^{\mathbf{M}}$, then $H$ will remain turbid in $\mathbf{N}$.

\begin{proof}
It is evident that if $\aleph_1^{\mathbf{M}} \neq \aleph_1^{\mathbf{N}}$, then the extension $\mathbf N$ collapses $\aleph_1$ to $\aleph_0$. In this case, $H$ in $\mathbf N$ is a countable subgroup of itself, and is free.

For the other direction, assume $\aleph_1^{\mathbf{M}} = \aleph_1^{\mathbf{N}}$. Then if $\{H_\alpha \mid \alpha < \aleph_1\}$ is an $\aleph_1$-filtration of $H$ in $\mathbf{M}$, it is also an $\aleph_1$-filtration of $H$ in $\mathbf{N}$. Moreover, by the absoluteness of $\aleph_1$-freeness,
\[ E := \{ \alpha < \aleph_1\mid H/{H_\alpha} \text{ is not } \aleph_1\text{-free} \}^{\mathbf{M}}=\{ \alpha < \aleph_1\mid H/{H_\alpha} \text{ is not } \aleph_1\text{-free} \}^{\mathbf{N}}, \]
and thus, $\Gamma(H)=[E]$ holds both in $\mathbf{M}$ and $\mathbf{N}$.
In particular, as $H$ is turbid in $\mathbf{M}$, there exists a club $C$ in $\mathbf{M}$ with $C \subseteq E$.
Note that $C$ is also a club in $\mathbf{N}$. Thus, $H$ remains turbid.
\end{proof}
\end{lemma}

\begin{remark}
Given an $\aleph_1$-free group $H$ with $\aleph_1$-filtration $\{H_\alpha \mid \alpha < \aleph_1\}$ and a transitive model extension with $\aleph_1^{\mathbf{N}} = \aleph_1^{\mathbf{M}}$, we note that while $E^{\mathbf{N}} = E^{\mathbf{M}}$ is true, $[E]^{\mathbf{N}} = [E]^{\mathbf{M}}$ may still fail.
\end{remark}

In the next section we demonstrate that if an $\aleph_1$-free group is not turbid, it can be forced to be free while preserving cardinals. Thus, we will have a characterization of turbid groups as those groups that cannot be made free without cardinal collapse.

\section{Forcing with Cardinal Preservation} \label{sec:3}

We have seen that an $\aleph_1$-free group $H$ of cardinality $\aleph_1$ is free if and only if $\aleph_1 -E$ contains a club. So, if we wish to force a non-free $\aleph_1$-free group $H$ of cardinality $\aleph_1$ to become free, we must force a club into $\aleph_1 - E$. We wish to do so while preserving $\aleph_1$.

\subsection{Forcing a Club into a Stationary Set} \label{sec:3.1}

Baumgartner, Harrington, and Kleinberg \cite{BHK} describe a forcing which, for any stationary subset $A$ of $\aleph_1$, forces a club into $A$. Furthermore, this forcing preserves reals and the cardinality of~$\aleph_1$. In particular, if $\mathbf{M}$ is a countable transitive model of {\rm ZFC}, and $A \subseteq \aleph_1$ is stationary, then there exists a generic extension $\mathbf{N}$ of $\mathbf{M}$ which has the same reals as $\mathbf{M}$ and in which there exists some club $C$ with $C \subseteq A$. The poset for this forcing consists of all closed subsets $p$ of $A$ of successor order type with $q \leq p$ if and only if $p$ is a subset of $q$ and $\min(q-p) > \max p$.

Alternatively, we may use a forcing by Abraham and Shelah \cite[Theorem 3]{AS}. The poset which describes this forcing has cardinality $\aleph_1$ and thus preserves all cardinalities. However, it does not preserve reals.

If $\Gamma(H) \neq [\aleph_1]$, then both of these forcings make $H$ free while preserving $\aleph_1$.

\begin{lemma}\label{not turbid}
Let $\mathbf{M}$ be a countable transitive model of {\rm ZFC} and in $\mathbf{M}$, let $H$ be an $\aleph_1$-free group of cardinality $\aleph_1$ with $\Gamma(H) \neq [\aleph_1]$. Then there exists a cardinality-preserving generic extension $\mathbf{N}$ of $\mathbf{M}$ with $H$ free in $\mathbf{N}$.

\begin{proof}
Let $\{H_\alpha \mid \alpha < \aleph_1\}$ be an $\aleph_1$-filtration of $H$ in  $\mathbf{M}$.
If $\aleph_1-E$ is not stationary, then there exists some club $C$ with $C \subseteq E$. Thus, $H$ is turbid, a contradiction.

So, $\aleph_1-E$ is stationary, and thus we can use \cite{AS} to produce a cardinality-preserving generic extension $\mathbf{N}$ of $\mathbf{M}$ in which $\aleph_1 -E$ contains a club.
As $E^{\mathbf{N}} = E^{\mathbf{M}}$, as seen in the proof of Lemma \ref{mtheorema}, $H$ is free in $\mathbf{N}$.
\end{proof}
\end{lemma}

Combining Lemmas \ref{mtheorema} and \ref{not turbid} gives the following necessary and sufficient condition under which an $\aleph_1$-free group of cardinality $\aleph_1$ can be forced to be free while preserving $\aleph_1$.

\begin{theorem}
Let $\mathbf{M}$ be a countable transitive model of {\rm ZFC} and $H$ an $\aleph_1$-free group of cardinality $\aleph_1$ in $\mathbf{M}$. Then there exists some transitive model extension $\mathbf{N}$ of $\mathbf{M}$ in which $\aleph_1$ is preserved and $H$ is free if and only if $\Gamma(H) \neq [\aleph_1]$ in $\mathbf{M}$.
\end{theorem}

Given any $\aleph_1$-free group $H$ of cardinality $\aleph_1$, we might equally well use forcing to add a club to $E$. This makes for a noteworthy variation on Lemma \ref{not turbid}.

\begin{lemma}
Let $\mathbf{M}$ be a countable transitive model of {\rm ZFC} and in $\mathbf{M}$, let $H$ be an $\aleph_1$-free group of cardinality $\aleph_1$ that is neither free nor turbid. Then there exist cardinality-preserving generic extension $\mathbf{N}_1$ and $\mathbf{N}_2$ of $\mathbf{M}$ with $H$ free in $\mathbf{N}_1$ and turbid in $\mathbf{N}_2$.
\end{lemma}

\subsection{The Partial Filtration Construction} \label{sec:3.2}

The arguments of the previous section use the $\Gamma$-invariant to reduce a question about $\aleph_1$-free groups to a purely set-theoretic problem about adding clubs to stationary sets. However, while this method proves effective, it provides very little additional insight, and one may desire forcing posets which are more algebraically motivated and closer to the original problem at hand. The following simple argument provides a first attempt at preserving the algebraic nature of the problem, translating the original poset on a stationary set back into a poset on the $\aleph_1$-filtration of the group $H$.

Let {\bf M} be a countable transitive model of {\rm ZFC}, and let $H\in {\bf M}$ be an $\aleph_1$-free group of cardinality $\aleph_1$ that is not turbid. Fix an $\aleph_1$-filtration $\mathcal F = \{H_\alpha \mid \alpha < \aleph_1\}$ of $H$ in {\bf M}, and let $E$ be defined as in Definition \ref{gammainvariant}. We have $\Gamma(H) =[E] \neq [\aleph_1]$, where
$\aleph_1 - E = \{\alpha < \aleph_1 \mid H/H_\alpha \text{ is } \aleph_1\text{-free}\}$ is stationary.

Now consider the set of all partial filtrations of $\mathcal F$ as defined in Definition \ref{filtration}. With a nod to \cite{BHK},
within this collection of partial filtrations we define $\mathcal P_{\rm pf}$ to be the set of all partial filtrations $\mathcal F_I$ of $\mathcal F$ such that for every $\alpha \in I$, $H/H_\alpha$ is $\aleph_1$-free.
We order $\mathcal P_{\rm pf}$ by ``end extension,'' that is, given two partial filtrations $\mathcal F_I$ and $\mathcal F_J$ we define $\mathcal F_J \le F_I$ if and only if $I \subseteq J$ and $\min(J-I) > \max I$.

Evidently, $(\mathcal P_{\rm pf},\le)$ is isomorphic to the Baumgartner, Harrington, Kleinberg forcing \cite{BHK} which rather uses the set $\{I \mid \mathcal F_I \in \mathcal P_{\rm pf}\}$ ordered under end extension and which adds a club $C$ to $\aleph_1-E$. Consequently, any $\mathcal P_{\rm pf}$-generic filter $G$ will add the $\aleph_1$-filtration $\mathcal F_C = \{H_\alpha \mid \alpha \in C\}$ of $H$ to ${\bf M}$. As we have $H_\beta \sqsubseteq H_\alpha$ for all $\beta\le \alpha <\aleph_1$, the filtration $\mathcal F_C$ gives evidence to a basis of $H$ in ${\bf M}[G]$, and $H$ is free.

\subsection{Forcing with Partial Bases} \label{sec:3.3}

Now let us introduce a poset which adds a basis to our non-turbid $\aleph_1$-free group $H$ directly. Since the final goal is to add a linearly independent spanning set, it is appropriate that our forcing poset should consist of ``partial bases,'' i.e., sets of linearly independent elements of $H$. In particular, we want to consider the following collection of linearly independent subsets of $H$:
\begin{equation}\label{P2}
\mathcal P_{\rm pb} = \{p \subseteq H \mid p \text{ is linearly independent}, |p| < \aleph_1, \text{ and } H/\langle p \rangle \text{ is } \aleph_1\text{-free}\}
\end{equation}
ordered under set inclusion, $q \le p$ iff $p \subseteq q$.

We start with a first easy observation that forcing with $(\mathcal P_{\rm pb},\le)$ always adds a new basis of $H$ to the base model $\mathbf{M}$.

\begin{theorem} \label{freeforce}
Let $\mathbf M$ be a countable transitive model of ${\rm ZFC}$, and let $H$ be an $\aleph_1$-free group of cardinality $\aleph_1$. Let $G$ be a $\mathcal P_{\rm pb}$-generic filter and define $\mathcal B = \bigcup G$. Then $\mathcal B$ is a basis of $H$ in $\mathbf{M}[G]$, with $\mathcal B \notin \mathbf{M}$.
\end{theorem}

\begin{proof}
It is evident that $\mathcal B$ is a linearly independent set since every finite subset of $B$ must be contained in some $p \in \mathcal P_{\rm pb}$. To show that $\mathcal B$ is a spanning set for $H$, suppose that $x \in H$, and define
\[D_x = \{p \in \mathcal P_{\rm pb} \mid x \in \langle p \rangle\}.\]

\begin{claim}
$D_x$ is dense in $\mathcal P_{\rm pb}$.
\end{claim}

Say that $q \in \mathcal P_{\rm pb}$. By a cardinal collapse, there is some generic extension $\mathbf M \subseteq \mathbf{N}$ such that $H/\langle q \rangle$ is countable in $\mathbf{N}$. Because $\aleph_1$-freeness is an absolute property \cite{HP}, $H/\langle q \rangle$ is a countable $\aleph_1$-free group in $\mathbf N$, thus free. Thus $H = \langle q \rangle \oplus C$, where $C$ is a free group with basis $\mathcal B_C$. Write $x = x_q + x_C$ according to this decomposition. Let $B'$ be a finite subset of $\mathcal B_C$ such that $x_C \in \langle B' \rangle$, and define $p = q \cup B'$. Then clearly $H/\langle p \rangle$ is free (thus $\aleph_1$-free) in $\mathbf N$. Absoluteness gives that $H/\langle p \rangle$ is also $\aleph_1$-free in $\mathbf M$, and $p \in \mathcal P_{\rm pb}$. Clearly $p \in D_x$ and $p \leq q$.\medskip

As $D_x$ is dense in $\mathcal P_{\rm pb}$, it follows that $x \in \langle \mathcal B \rangle$, and $\mathcal B$ is a basis for $H$ in $\mathbf M[G]$. In addition $\mathcal B\notin \mathbf{M}$. This is immediate if $H$ is not free in $\mathbf{M}$. However, note that for $\mathcal B\in \mathbf{M}$, the set
\[ D = \{p \in \mathcal P_{\rm pb} \mid p \not\subseteq \mathcal B\}\]
would be dense in $\mathcal P_{\rm pb}$, leading to a contradiction. Thus, $\mathcal B\notin \mathbf{M}$ even for free~$H$.
\end{proof}

The next question is whether $(\mathcal P_{\rm pb},\le)$ preserves $\aleph_1$. In Lemma \ref{mtheorema}, we have already seen that $\aleph_1$ must collapse whenever $H$ is turbid. What happens when $H$ is non-turbid? While not conclusive evidence, it turns out that $(\mathcal P_{\rm pb},\le)$ always fails the closure condition for cardinal preservation.

\begin{proposition}\label{omegaclosed}
If $H$ is an $\aleph_1$-free group of cardinality $\aleph_1$, then $(\mathcal P_{\rm pb},\le)$ fails to be $\aleph_1$-closed.
\end{proposition}

Before we prove Proposition \ref{omegaclosed}, we start investigating a simplified situation.

\begin{lemma}
Let $H = \bigoplus_{n \in \omega} \mathbb Z e_n$, then $(\mathcal P_{\rm pb},\le)$ is not $\aleph_1$-closed.
\end{lemma}

\begin{proof}
We define a sequence of elements \[s_i = e_i - (i+1) e_{i+1} \in H\] and sets $S_j = \{s_i \mid i < j\}\subseteq H$ for any $j \in \omega =\{0,1,2,\ldots\}$. A straightforward calculation shows that $S_j$ is linearly independent for each $j \in\omega$. We define a sequence of sets such that $\mathcal B_0 = \{e_0, e_1, \ldots\}$ and for all $j>0$, $\mathcal B_j = \{s_0, s_1, \ldots ,s_{j-1}, e_{j}, e_{j+1}, \ldots\}$. Clearly, each of these sets is a basis for $H$ since it is linearly independent and spanning.
Thus $\langle S_j \rangle$ is a pure subgroup of $H$ for all $j$ since
\[H = \langle S_j \rangle \oplus \bigoplus_{i\ge j} \mathbb Z e_i,\] and summands are pure in free groups. Moreover, $H/\langle S_j \rangle$ is free for all $j$. Hence $S_j \in \mathcal P_{\rm pb}$ for all $j \in \omega$.

We next define $S = \{s_i \mid i \in \omega\}$. Note that the sets $S_j \subseteq S$ and the $S_j$ form a countable descending sequence in $\mathcal P_{\rm pb}$ with limit $S$. Furthermore $e_0 \notin \genby S$, so that $\langle S \rangle$ is a proper subgroup of $H$ of corank 1. To see this, define a map $\phi: H \rightarrow \mathbb Q$  which maps $e_i \mapsto \frac{1}{i!}$, and extend it linearly to make a homomorphism. Then $\phi(s_i) = \phi(e_i) - (i+1) \phi(e_{i+1}) = 0$, and we have $\ker \phi = \genby S$. However $\phi(e_0) = 1$ so that $e_0 \notin \ker \phi$. Thus, $e_0 + \genby S$ is a nonzero element of $H/\genby S$.

To show that $S \notin \mathcal P_{\rm pb}$, it suffices to verify that $H/\genby S$ is torsionfree but $e_0 + \genby S$ is divisible. Thus, suppose that $n(h + \genby S) = 0$. Then this implies that there exist $n_i\in \mathbb{Z}, s_i \in S$ such that $nh = \sum_{\, i} n_i s_i$. There must exist some $k\in\omega$ such that $s_i \in S_k$ for all $i$ in this expression. As $\genby{S_k}$ is pure in $H$, $h \in \genby{S_k} \subseteq \genby S$. Hence $h + \genby S = \genby S$ in $H/\genby S$. On the other hand, our telescopic construction of the $s_i$ implies that
\[e_0 \equiv i!\, e_i \mod \genby S\] for all $i \in \omega$. Thus $e_0+ \genby S$ is divisible in $H/\genby S$, and the torsionfree group $H/\genby S$ contains a copy of $\mathbb Q$. So $S \notin \mathcal P_{\rm pb}$.

All that is left to check is that there is no element $p \in \mathcal P_{\rm pb}$ such that $p \leq S$. If there did exist such an element, $H/\langle p \rangle$ would be $\aleph_1$-free. As $S$ is a subset of $p$, $\langle p \rangle/\langle S \rangle$ is a free group. Hence $H/\langle S \rangle$ is $\aleph_1$-free by the transitivity of $\aleph_1$-freeness. But this would imply that $S \in \mathcal P_{\rm pb}$, a final contradiction. Hence $(\mathcal P_{\rm pb},\le)$ is not $\aleph_1$-closed.
\end{proof}

This basic idea can be implemented into a proof of Proposition \ref{omegaclosed}.

\begin{proof}
Suppose that $H$ is an $\aleph_1$-free group of cardinality $\aleph_1$. In a first step we are going to construct a sequence $\{h_i \mid i\in \omega \}$ of nonzero elements of $H$. Start with an arbitrary $0\ne h_0 \in H$. Given $\{ h_j \mid j\le i\}$, let $H_i =\langle h_j \mid j\le i\rangle_*$ denote the purification in $H$, and choose $h_{i+1} \in H-H_i$ with respect to a fixed well-ordering of $H$. Note that $\{h_i \mid i\in \omega \}$ is linearly independent. Hence, $H_i$ is a group of rank~$i$. Moreover, with Pontryagin's criterion we can show that $H/H_i$ is $\aleph_1$-free.

As $H$ is $\aleph_1$-free, $H_i$ must be free of rank $i$. Moreover, $H_i \sqsubseteq H_{i+1}$.
Thus, we can define a sequence $q=\{e_i \mid i\in \omega \}$ of nonzero elements of $H$ such that
$e_0=h_0$ and $H_{i+1} =H_i \oplus \mathbb Z e_{i+1}$. In particular, for $q_j=\{e_i \mid i\le j \}$ holds $H_j =\genby{q_j}$, and $\{q_j \mid j \in \omega\}$ is a strictly descending sequence in $(\mathcal P_{\rm pb},\le)$.

If $q=\bigcup_{j\in \omega} q_j \notin \mathcal P_{\rm pb}$, then the final paragraph of the previous proof shows that there does not exists any $p \in \mathcal P_{\rm pb}$ such that $p \le q$.

All that remains to be checked is the case where $q=\bigcup_{j\in \omega} q_j \in \mathcal P_{\rm pb}$. For this,
we will consider $\langle q \rangle = \bigoplus_{i \in \omega} \mathbb Z e_i \subseteq H$. Let
\[s_i = e_i - (i+1)e_{i+1} \in \langle q \rangle\]
for all $i \in \omega$, and define the sets $S_j$ and $S$ as in the previous proof. Then, with $H/\genby q$ $\aleph_1$-free and $\genby q/\genby{S_j}$ free, $H/\genby {S_j}$ is $\aleph_1$-free by transitivity, while $\genby q/ \genby S \subseteq H/\genby S$ contains a nontrivial divisible subgroup. Hence $S_j \in \mathcal P_{\rm pb}$ and $S = \bigcup_{j\in \omega} S_j \notin \mathcal P_{\rm pb}$, and $(\mathcal P_{\rm pb},\le)$ fails to be $\aleph_1$-closed as in the previous lemma.
\end{proof}

Cardinal preservation requires a more in-depth argument which is inspired by~\cite{BHK}.

\begin{theorem}\label{MainTheorem}
Let $\mathbf M$ be a countable transitive model of {\rm ZFC} with
\[ \mathbf M \vDash \text{``$H$ is a non-turbid $\aleph_1$-free group of cardinality $\aleph_1$,''}\]
and let $G$ be any $\mathcal P_{\rm pb}$-generic filter. Then $\mathbf{M}[G]$ preserves $\aleph_1$.
\end{theorem}

\begin{proof}
Fix some $\aleph_1$-filtration $\{H_\alpha \mid \alpha < \aleph_1\}$ of $H$. Without loss of generality, we may assume that $H_\alpha \subsetneq H_\beta$ for all $\alpha < \beta$. Let $E := \{\alpha < \aleph_1 \mid H/ H_\alpha \text{ is not } \aleph_1\text{-free}\}$. Since $H$ is not turbid, $\aleph_1 - E = \{\alpha < \aleph_1 \mid H/H_\alpha \text{ is } \aleph_1 \text{-free}\}$ is stationary. Then define
\[\mathcal P' := \{p \mid p \text{ is a basis of } H_\alpha \mbox{ for some } \alpha \in \aleph_1 - E\} \subseteq \mathcal P_{\rm pb}.\]

\begin{claim}
$\mathcal P'$ is dense in $\mathcal P_{\rm pb}$.
\end{claim}

To see this, let $p \in \mathcal P_{\rm pb}$. As $p$ is countable, there exists some $\alpha$ such that $p \subseteq H_\alpha$. Since $\aleph_1 - E$ is stationary (hence unbounded), without loss of generality, assume that $\alpha \in \aleph_1 - E$. Now, by assumption $H/ \langle p \rangle$ is $\aleph_1$-free. Thus, as $H_\alpha / \langle p \rangle$ is a countable subgroup of $H/ \langle p \rangle$, it follows that $H_\alpha/\langle p \rangle$ is free, and $p$ can be extended to a basis $q$ of $H_\alpha$. We have $q \in \mathcal P'$ and $q \leq p$, so density follows. \medskip

Let $p \in \mathcal P'$. Since our $\aleph_1$-filtration is strictly increasing, there exists some unique ordinal $h(p)<\aleph_1$ such that $\langle p \rangle = H_{h(p)}$, which will be called the \emph{height} $h(p)$ of $p$.

Towards a contradiction, suppose that $\mathbf{M}[G]$ contains a bijection $f: \omega \rightarrow \omega_1^{\mathbf M}$ which is not contained in $\mathbf M$. Then let $\tau$ be  a $\mathcal P_{\rm pb}$-name such that $\tau$ evaluates to $f$ on $G$. Hence there is some $p \in \mathcal P_{\rm pb}$ (hence some $p \in \mathcal P'$) such that
\begin{equation} \label{eq1}
p \Vdash \text{``}\tau \text{ is a bijection between } \widecheck \omega \text{ and } \widecheck{\omega_1^{\mathbf M}} \text{ with } \tau \notin \widecheck{(\leftidx{^{\omega}}\omega_1)^{\mathbf M}}.\text{''}
\end{equation}

In order to approximate this new bijection $f$ within $\mathbf M$, we define an ascending sequence $\{A_\alpha \mid \alpha <\aleph_1\}$ of subsets $A_\alpha \subseteq \mathcal P' \times \omega \times \leftidx{^{<\omega}}\omega_1$ in $\mathbf M$ such that when $(q,n,g) \in A_\alpha$ then $\dom(g) = n$ and
\begin{equation} \label{eq2}
\text{$q \Vdash$ ``$\tau$ is a bijection  between $\widecheck \omega$ and $\widecheck{\omega_1^{\mathbf M}}$ with $\tau \notin \widecheck{(\leftidx{^{\omega}}\omega_1)^{\mathbf M}}$ and $\tau \upharpoonright \widecheck n = \widecheck g$.''}
 \end{equation}
In other words, $A_\alpha$ collects some partial bijections $g$ which approximate our new bijection $f$, and elements of $\mathcal P'$ which force these approximations. To properly scaffold this construction, for each $A_\alpha$ we introduce the ordinal
\[h_\alpha = \sup\{h(q) \mid (q,n,g) \in A_\alpha\}.\]
We now inductively construct $\{A_\alpha \mid \alpha <\aleph_1\}$ as follows:\smallskip

First define $A_0 = \{(p, 0, \emptyset)\}$ where $p$ is the element of $\mathcal P'$ defined by \eqref{eq1}. For the successor step, let $A_\alpha$ be defined, assume $h_\alpha < \aleph_1$, and let $(q,n,g) \in A_\alpha$. Choose a $\mathcal P_{\rm pb}$-generic filter $G'$ with $q \in G'$. As $(q,n,g)$ satisfies \eqref{eq2}, within $\mathbf M[G']$, $\tau$~evaluates to a bijection $\tau_{G'}$ between $\omega$ and ${\omega_1^\mathbf{M}}$ with $\tau_{G'} \upharpoonright n = g$. Define a function $g'\in \leftidx{^{n+1}}\omega_1$ by $g' = \tau_{G'} \upharpoonright (n+1)= g \cup \{ (n, \tau_{G'}(n)\}$. By the Fundamental Theorem of Forcing there exists some $q' \in G'$ such that
\[\text{$q' \Vdash$ ``$\tau$ is a bijection between $\widecheck \omega$ and $\widecheck{\omega_1^{\mathbf M}}$ with $\tau \notin \widecheck{(\leftidx{^{\omega}}\omega_1)^{\mathbf M}}$ and $\tau \upharpoonright \widecheck{n+1} = \widecheck{g'}$.''}\]
Because $q$ and $q'$ are contained in the filter $G'$, we can choose a common extension $q'' \in G'$ (hence some $q'' \in \mathcal P'$ with $h(q'') > h_\alpha$) of $q$ and $q'$. Then $(q'',n+1,g')$ satisfies
\begin{equation} \label{eq3}
\text{$q'' < q$, $h(q'') > h_\alpha\ge h(q)$, $g \subseteq g'$,}
 \end{equation}
 and
 \begin{equation} \label{eq4}
\text{$q''\! \Vdash$\!\! ``$\tau$ is a bijection between $\widecheck \omega$ and $\widecheck{\omega_1^{\mathbf M}}$ with $\tau \notin \widecheck{(\leftidx{^{\omega}}\omega_1)^{\mathbf M}}$ and $\tau\! \upharpoonright \widecheck{n\!+\!1} = \widecheck{g'}$.''}
 \end{equation}
Now that the existence of such a $(q'',n+1,g')$ is settled, we may rather choose $(q'',n+1,g')\in \mathcal P' \times \omega \times \leftidx{^{<\omega}}\omega_1$ with properties \eqref{eq3} and \eqref{eq4} to be minimal with respect to some fixed well-ordering of $\mathcal P' \times \omega \times \leftidx{^{< \omega}}\omega_1$ in {\bf M}. This process is well-defined and decidable within {\bf M}, since $\Vdash$ may be replaced by $\Vdash^\ast$. Then define
\[A_{\alpha+1} = A_\alpha \cup \{(q'', n+1, g') \mid (q,n,g) \in A_\alpha\},\]
completing the successor step. Finally, for $\alpha$ a limit ordinal define $A_\alpha = \bigcup_{\beta < \alpha} A_\beta$.\smallskip

From the construction, it is clear that $|A_0|=1$ and $|A_{\alpha+1}-A_\alpha|\le |A_\alpha|$. An easy transfinite induction shows $A_\beta \subseteq A_\alpha$ for $\beta \le \alpha < \aleph_1$. Furthermore, $A_\alpha$ is countable for all $\alpha < \aleph_1$. Thus, $\{h(q) \mid (q,n,g) \in A_\alpha\} \subseteq \aleph_1$ is countable, justifying the initial assumption $h_\alpha < \aleph_1$ of our successor step construction.

Finally, consider the set of heights $C = \{h_\alpha \mid \alpha < \aleph_1\}$. From the construction, it is clear that $h_\beta < h_\alpha <\aleph_1$ for $\beta < \alpha < \aleph_1$, and $C\subseteq \aleph_1$ is a club. Define the club
\[C^* =\{h_\alpha \mid \alpha<\aleph_1 \text{ is a limit ordinal}\}\subseteq \aleph_1.\]

Since $\aleph_1 - E$ is a stationary set, we have $(\aleph_1 - E) \cap C^*\ne \emptyset$. Choose some $h_{\alpha^*} \in (\aleph_1 - E) \cap C^*$. Since $\alpha^*$ is a countable limit ordinal, there is a strictly increasing sequence of ordinals $(\alpha_n)_{n \in \omega}$ with supremum $\alpha^*$. Thus
\[h_{\alpha^*} = \sup_{\beta < \alpha^*} h_\beta = \sup_{n \in \omega} h_{\alpha_n}.\]

From $\{A_\alpha \mid \alpha <\aleph_1\}$ and $(\alpha_n)_{n \in \omega}$, we now inductively define for each $n\in \omega$ some $(q_n, n, g_n)\in \mathcal P' \times \omega \times \leftidx{^{<\omega}}\omega_1$. First, set $(q_0, 0, g_0) = (p, 0, \emptyset) \in A_0 \subseteq A_{\alpha_0}$. Then, assuming that $(q_n, n, g_n) \in A_{\alpha_n}$ is given, let
 \begin{equation} \label{eq5}
(q_{n+1}, n+1, g_{n+1}) = (q''_{n}, n+1, g'_n) \in A_{\alpha_n +1} - A_{\alpha_n} \subseteq A_{\alpha_{n+1}}
 \end{equation}
as detailed in the previous construction. With \eqref{eq3}, we have $q_{n+1} = q_n'' < q_n$ and $g_n \subseteq g_n' = g_{n+1}$. Then $\{q_n \mid n \in \omega\}$ is a strictly descending sequence in $(\mathcal P_{\rm pb},\le)$ and $g = \bigcup_{n \in \omega} g_n$ defines a function $g: \omega \rightarrow \omega_1$ in $\mathbf M$. Finally, with \eqref{eq3} and \eqref{eq5}, we have $h_{\alpha_n} < h(q_n'') = h(q_{n+1}) \leq h_{\alpha_{n+1}}$. Thus $\sup_{n \in \omega}h(q_n) = \sup_{n \in \omega} h_{\alpha_n} = h_{\alpha^*}$.

Now note that $q_n \in \mathcal P'$ means that $q_n$ is a basis for the subgroup $H_{h(q_n)} \subseteq H$ in our $\aleph_1$-filtration. Then $q^* := \bigcup_{n \in \omega} q_n$ is a basis for $H_{\sup_{n \in \omega} h(q_n)} = H_{h_{\alpha^*}}$. Since $h_{\alpha^*} \in \aleph_1 - E$, we have $q^* \in \mathcal P'$ with $q^* \leq q_n$ for all $n\in \omega$. With \eqref{eq4}, we have
\[\text{$q^* \Vdash ``\tau$ is a bijection between $\widecheck \omega$ and $\widecheck{\omega_1^{\mathbf M}}$ with $\tau \notin \widecheck{(\leftidx{^{\omega}}\omega_1)^{\mathbf M}}$ and $\tau = \widecheck g$.''}\]
Thus, if $G^*$ is a $\mathcal P_{\rm pb}$-generic filter with $q^* \in G^*$, then $\tau$ evaluates to the bijection $g: \omega \rightarrow \omega_1$ in $\mathbf M$ which contradicts $\tau_{G^*} \notin \mathbf M$.
\end{proof}

\begin{remark}
In addition to preserving $\aleph_1$, forcing with $\mathcal P_{\rm pb}$ also preserves reals. To see this, in the above proof we can replace
\[\mbox{``\,}\tau \text{ is a bijection between } \widecheck \omega \text{ and } \widecheck{\omega_1^{\mathbf M}} \text{ with } \tau \notin \widecheck{(\leftidx{^{\omega}}\omega_1)^{\mathbf M}}\mbox{''}\]
with
\[\mbox{``\,}\tau \text{ is a function between } \widecheck \omega \text{ and } \widecheck 2 \text{ with } \tau \notin \widecheck{(^\omega 2)^{\mathbf M}}\mbox{.''}\]
\end{remark}

\section{Open Questions and Future Directions}

Our results on turbid groups and forcing posets are tailored specifically to add a basis to an $\aleph_1$-free group. Extending these results gives rise to some interesting questions for future research.

\subsubsection*{Characterizing Turbid Groups}

While the $\aleph_1$-free groups $H$ of cardinality $\aleph_1$ with $\Gamma(H)=[\emptyset]$ coincide with the free groups of cardinality $\aleph_1$, we are still missing a similarly succinct algebraic characterization for $\aleph_1$-free groups $H$ of cardinality $\aleph_1$ with $\Gamma(H)=[\aleph_1]$. Corollary \ref{charturbid}(b) is the best we have and a more intuitive algebraic characterization would promote a better understanding of turbid groups.

\begin{problem}
Provide a simple algebraic characterization for turbid groups that avoids the use of $\aleph_1$-filtrations and $\Gamma$-invariants.
\end{problem}

\subsubsection*{Groups of Larger Cardinality}

While Theorem \ref{mtheorem} provides a full answer to the question when a basis can be added to an $\aleph_1$-free group of cardinality $\aleph_1$ without cardinal collapse, this question still remains open for $\aleph_1$-free group of larger cardinality.

\begin{problem}
Let $\mathbf M$ be a transitive model of {\rm ZFC} with \[ \mathbf M \vDash \text{``$H$ is an $\aleph_1$-free group of size $\kappa>\aleph_1$.\!''}\]
Provide necessary and sufficient algebraic conditions on $H$ if $H$ cannot be free in any cardinality preserving transitive model extension $\mathbf N$ of $\mathbf M$.
\end{problem}

Alternatively, one might rather ask for necessary and sufficient conditions if $H$ cannot be free in any cardinality preserving forcing extension of a countable transitive ground model $\mathbf M$. Of course, adding a basis to $H$ results in a cardinal collapse if $H$ has a turbid subgroup. Are there any other restrictions?

This problem will need an adequate replacement for the $\Gamma$-invariant which is fit for groups of size $\kappa>\aleph_1$. It should be noted that the $\Gamma$-invariant easily extends to $\aleph_n$-free groups of cardinality $\aleph_n$ under ZFC and $\kappa$-free groups of cardinality $\kappa$ for regular $\kappa \ge \aleph_1$ under V=L. However, the defining set
\[E=\{\alpha < \kappa \mid H/{H_\alpha} \text{ is not } \kappa \text{-free} \}\]
will not be absolute anymore.

\subsubsection*{Splitting Off Direct Summands}

Closely related to adding a basis to an $\aleph_1$-free group of cardinality $\aleph_1$ is the task of splitting off a subgroup $K$ of an $\aleph_1$-free group $H$ of cardinality $\aleph_1$. Clearly, by the absoluteness of $\aleph_1$-freeness, $H/K$ being $\aleph_1$-free is a necessary requirement. If $|H/K| <\aleph_1$, then $H/K$ is free and $K\sqsubseteq H$ holds automatically in the ground model $\mathbf M$. If $H/K$ is non-turbid of cardinality~$\aleph_1$, Theorem~\ref{mtheorem} provides a cardinality preserving forcing extension $\mathbf N$ where $H/K$ is free and $K\sqsubseteq H$. This leaves us with the following intriguing question.

\begin{problem}
Let $\mathbf M$ be a countable transitive model of {\rm ZFC} with
\[ \mathbf M \vDash\! \text{``$H$\! is an $\aleph_1$-free group of size $\aleph_1$\! and $K\subseteq H$\! is a subgroup with $H/K$\! turbid.\!''}\]
Does there exist a cardinality preserving forcing extension $\mathbf N$ with $K\sqsubseteq H$ in $\mathbf N$?
\end{problem}

In particular, we are looking for a minimally invasive forcing extension that turns $K$ into a direct summand of $H$ without adding any basis to either summand.\medskip

All these questions around turbid groups and $\aleph_1$-free forcings are the subject of ongoing research.


\begin{thebibliography}{99}\label{lit} \markright{}
\bibitem{AS} U. Abraham, S. Shelah, {\it Forcing closed unbounded sets}, J. Symb. Log. {\bf 48} (1983), 643--657.
\bibitem{AGS} M. Asgharzadeh, M. Golshani, S. Shelah, {\it Graphs represented by $\operatorname{Ext}$} (2021), arxiv:2110.11143.
\bibitem{Baer37} R. Baer, {\it Abelian groups without elements of finite order}, Duke Math. J {\bf 3}, 68--122.
\bibitem{BHK} J.H.~Baumgartner, L.A.~Harrington, E.M.~Kleinberg, {\it Adding a closed unbounded set}, J. Symb. Log. {\bf 41} (1976), 481--482.
\bibitem{BHP2} D.~Bossaller, D.~Herden, A.V.~Pasi,
{\it The Baer-Specker group is worse beyond failing strongly $\aleph_1$-free}, in preparation (2022).
\bibitem{BS} G. Braun, Lutz Str\"ungmann,
{\it The independence of the notions of Hopfian and co-Hopfian abelian $p$-groups},  Proc. Amer. Math. Soc. {\bf 143} (2015), 3331--3341.
\bibitem{CG} A.L.S. Corner, R. Göbel, {\it Prescribing endomorphism algebras -- A unified treatment}, Proc. London Math. Soc. {\bf 50} (1985), 447--479.
\bibitem{DG} M. Dugas, R. Göbel, {\it Every cotorsion-free algebra is an endomorphism algebra,} Math. Z. {\bf 181} (1982), 451--470.
\bibitem{Eklof76} P.C. Eklof, \emph{Whitehead's problem is undecidable}, Amer. Math. Monthly {\bf 83} (1976), 775--788.
\bibitem{Eklof77} P.C. Eklof, {\it Methods of logic in abelian group theory}, Lecture Notes in Math. {\bf 616} (1977), Springer, 251--269.
\bibitem{EM} P.C. Eklof, A.H. Mekler, {\it Almost-Free Modules: Set-Theoretic Methods}, North Holland, Amsterdam (2002).
\bibitem{Fuchs} L. Fuchs, {\it Abelian Groups}, Springer, London (2015).
\bibitem{GT} R. G\"obel, J. Trlifaj, \emph{Endomorphism Algebras and Approximations of Modules -- Vol. {\bf 1, 2}},
Expositions in Mathematics {\bf 41}, Walter de Gruyter Verlag, Berlin (2012).
\bibitem{HP} D. Herden, A.V. Pasi, {\it On the absoluteness of $\aleph_1$-freeness}, arxiv:2104.10341 (2021).
\bibitem{Krueger} J. Krueger, {\it Destroying stationary sets}, Israel J. Math. {\bf 147} (2005), 285--328.
\bibitem{MaS} M. Magidor, S. Shelah, {\it When does almost free imply free? (For groups, transversals, etc.)},  J. Amer. Math. Soc. {\bf 7} (1994), 769--830.
\bibitem{MeS} A. Mekler, S. Shelah, {\it Every coseparable group may be free}, Israel J. Math. {\bf 81} (1993), 161--178.
\bibitem{S1} S. Shelah, {\it Infinite abelian groups, Whitehead problem and some constructions}, Israel J. Math. {\bf 18} (1974), 243--256.
\bibitem{S2} S. Shelah, {\it Quite free complicated abelian groups, pcf and black boxes}, Israel J. Math. {\bf 240} (2020), 1--64.
\bibitem{specker50} E. Specker, {\it Additive Gruppen von Folgen ganzer Zahlen}, Port. Math. {\bf 9} (1950), 131--140.
\end{thebibliography}
\end{document}